\documentclass[12pt,a4paper]{amsart}
\usepackage[utf8]{inputenc}

\usepackage{amssymb,amsmath,amsthm}
\usepackage[abbrev,alphabetic]{amsrefs}
\usepackage{tikz-cd}
\usetikzlibrary{arrows}
\usepackage{graphicx}
\usepackage[hidelinks]{hyperref}
\usepackage{mathtools}
\usepackage[titletoc,title]{appendix}
\usepackage{graphicx}
\usepackage{caption}
\usepackage{subcaption}
\usepackage{blindtext}
\usepackage[titletoc,title]{appendix}
\usepackage[english]{babel}
\usepackage{mathrsfs}
\usepackage{microtype}
\usepackage{xcolor}
\usepackage{stmaryrd}

\usepackage[capitalize]{cleveref}

\newtheorem{theorem}{Theorem}[section]
\newtheorem{lemma}[theorem]{Lemma}
\crefname{lemma}{Lemma}{Lemmata}
\newtheorem{corollary}[theorem]{Corollary}
\newtheorem{proposition}[theorem]{Proposition}
\theoremstyle{definition}

\newtheorem*{theorem*}{Theorem}
\newtheorem*{conjecture*}{Conjecture}

\newcommand{\isom}{\cong}

\newcommand{\Z}{\mathbb{Z}}
\newcommand{\Q}{\mathbb{Q}}

\newcommand{\C}{\mathbb{C}}

\def\injects{\hookrightarrow}

\DeclareSymbolFontAlphabet{\amsmathbb}{AMSb}

\DeclareMathOperator{\cd}{cd}

\title{Group rings of three-manifold groups}
\author{Dawid Kielak}
\author{Marco Linton}

\address[DK]{University of Oxford, Oxford, OX2 6GG, UK}

\email{kielak@maths.ox.ac.uk}

\address[ML]{University of Oxford, Oxford, OX2 6GG, UK}

\email{marco.linton@maths.ox.ac.uk}

\begin{document}

\begin{abstract}
Let $G$ be the fundamental group of a three-manifold. By piecing together many known facts about three manifold groups, we establish two properties of the group ring $\C G$. We show that if $G$ has rational cohomological dimension two, then $\C G$ is coherent. We also show that if $G$ is torsion-free, then $G$ satisfies the Strong Atiyah Conjecture over $\C$ and hence that $\C G$ satisfies Kaplansky's Zero Divisor Conjecture.
\end{abstract}	

\maketitle

\section{Introduction}

It is the purpose of this note to present two results on the group rings of three-manifold groups. The first result regards the coherence property: a ring $R$ is \emph{(left) coherent} if every finitely generated (left) ideal $I$ of $R$ is finitely presented. The coherence of fundamental groups of three-manifolds has long been known by the compact core theorem due to Scott \cite{Scott1973} and, independently, Shalen. On the other hand, the coherence of the group ring of a three-manifold group has not yet been considered. 

By piecing together many known properties of three-manifolds, we show the following.

\begin{theorem}
\label{manifold_free-by-cyclic}
If $G$ is a finitely generated fundamental group of a three-manifold with $\cd_{\Q}(G) <3$, then $G$ is virtually free-by-cyclic, where $\cd_{\Q}$ denotes the rational cohomological dimension.
\end{theorem}

Combining this with a result from \cite{JZL23}, we obtain coherence of the group ring for low dimensional three-manifold groups.

\begin{corollary}
\label{coherence}
Let $G$ be the fundamental group of a three-manifold and let $K$ be a field. If $\cd_{\Q}(G)<3$, then $KG$ is coherent.
\end{corollary}

The key new step we use to prove \cref{manifold_free-by-cyclic} is a `drilled' version of the standard doubling construction. We suspect this was previously known to experts, although no reference could be found. 

The second result regards the well-known Kaplansky Zero-divisor Conjecture which predicts that if $G$ is a torsion-free group and $K$ is a field, then $KG$ has no non-trivial zero-divisors. If a torsion-free group $G$ satisfies the Strong Atiyah Conjecture over $\C$, then its group ring $\C G$ embeds into a skew-field and thus cannot contain any non-trivial zero-divisors -- for the statement of the conjecture and related background, see the book of L\"uck~\cite{luck_02}*{Section 10}. Linnell established the Strong Atiyah Conjecture over $\C$ for groups with bounded torsion within a large class $\mathcal{C}$ of groups \cite[Theorem 1.5]{linnell_93}. The class $\mathcal{C}$ is defined to be the smallest class of groups containing all free groups that is closed under directed unions and extensions by elementary amenable groups.

The Strong Atiyah Conjecture, and thus also the Kaplansky Zero-divisor Conjecture, has been shown to hold for fundamental groups of admissible three-manifolds by Friedl--L{\"{u}}ck \cite{friedl_19}, who showed that all such groups  lie in $\mathcal{C}$. Recall that a three-manifold is admissible if it is connected, orientable, and irreducible, its boundary is empty or a disjoint union of tori and if its fundamental group is infinite. If one knew that the Strong Atiyah Conjecture was closed under free products and finite extensions, then this would suffice to extend Friedl--L{\"{u}}ck's result to apply to all finitely generated three-manifold groups. However, closure properties of this kind are not known to hold in general and are non-trivial to establish even within restricted classes of groups. For partial results in this direction, see \cite{schick_00,linnell_07,schreve_14}. In \cref{C_section} we show that $\mathcal{C}$ is closed under free products, which may be of independent interest. 

\begin{proposition}
	\label{C_free_products}
	Linnell's class $\mathcal{C}$ is closed under arbitrary free products.
\end{proposition}

Since $\mathcal{C}$ is closed under finite extensions, we may then use this in conjunction with \cref{manifold_free-by-cyclic} to extend Friedl and L{\"{u}}ck's result to all three-manifold groups.

\begin{theorem}
\label{main thm}
Fundamental groups of three-manifolds lie in Linnell's class $\mathcal C$. 
\end{theorem}

Combining this with Linnell's theorem, we obtain our second result on group rings of three-manifold groups.

\begin{corollary}
	\label{main cor}
Let $G$ be the fundamental group of a three-manifold.
\begin{enumerate}
\item If $G$ has bounded torsion (in particular, if $G$ is finitely generated), then $G$ satisfies the Strong Atiyah Conjecture over $\C$.
\item If $G$ is torsion-free, then $\C G$ has no non-trivial zero divisors and hence the Kaplansky Zero-divisor Conjecture holds for $\C G$.
\end{enumerate}
\end{corollary}

We remark that without the assumption of bounded torsion, the Strong Atiyah Conjecture is false in general by work of Grigorchuk--Linnell--Schick--\.{Z}uk \cite{grigorchuk_00}.

\subsection*{Acknowledgements}
The authors would like to thank the referee, whose comments led to a greatly improved exposition.
This work has received funding from the European Research Council (ERC) under the European Union's Horizon 2020 research and innovation programme (Grant agreement No. 850930).

The authors would like to thank Adele Jackson for insightful conversations on Seifert-fibred spaces, and Peter Shalen for comments on an earlier version of this article.

\section{Low dimensional three-manifold groups}

We start by listing some properties of free products.
The first two statements in the following result appear as \cite[Lemma 14 \& Theorem 15]{baumslag_09}. The third is a straightforward application of the Kurosh Subgroup Theorem.

\begin{proposition}
	\label{closure}
	\begin{enumerate}
		\item A free product of any number of free-by-$\Z$ groups is free-by-$\Z$.
		\item A free product of finitely many virtually free-by-cyclic groups is virtually free-by-cyclic.
		\item A free product of finitely many virtually torsion-free groups is virtually torsion-free.
	\end{enumerate}	
\end{proposition}

The following is a summary of well-known facts about three-manifold groups. We include a proof for the readers convenience.

\begin{proposition}
\label{facts}
Let $M$ be a connected three-manifold and let $G$ be its fundamental group. If $G$ is finitely generated, then there is a finitely generated free group $F$ and finitely many compact, connected, orientable, aspherical three-manifolds $M_1, \dots, M_n$, each with a (possibly trivial) incompressible boundary, such that a finite-index subgroup of $G$ is isomorphic to the free product $F*\left(*_{i=1}^n\pi_1(M_i)\right)$.
\end{proposition}

\begin{proof}
By replacing $M$ with its orientation double cover, we may assume that $M$ is orientable. The Compact Core Theorem allows us to assume that $M$ is compact. Since an orientable prime non-irreducible three-manifold is homeomorphic to $S^2\times S^1$, which has fundamental group $\Z$, we may combine the Prime Decomposition Theorem with the Loop Theorem (in the form of \cite[Lemma 1.4.2]{aschenbrenner_15}) to obtain a finitely generated free group $F$ and finitely many compact, connected, irreducible and orientable three-manifolds $M_1, \dots, M_k$ with (non-spherical and possibly empty) incompressible boundaries, such that $G\isom F*\left(*_{i=1}^k\pi_1(M_i)\right)$.

If $G$ is torsion-free, then each $M_i$ would also be aspherical by the Sphere Theorem. If $G$ is not torsion-free, then it suffices to show that $G$ is virtually torsion-free as then we can apply the above argument to an appropriate finite sheeted cover of $M$. Since each $M_i$ is either aspherical or covered by $S^3$, it follows that each $\pi_1(M_i)$ is either torsion-free or finite. Since a finite free product of virtually torsion-free groups is virtually torsion-free by \cref{closure}, it follows that $G$ is virtually torsion-free.
\end{proof}

In order to prove \cref{manifold_free-by-cyclic} we shall need a variation of the standard doubling construction.

\begin{proposition}[Drilled doubling construction]
\label{drilled}
Let $M$ be a compact aspherical three-manifold with non-empty incompressible boundary. Then $\pi_1(M)$ is a retract of the fundamental group of a compact aspherical three-manifold with non-empty purely toroidal boundary.
\end{proposition}

\begin{proof}
Assuming $\pi_1(M)$ is non-trivial, there is some embedded curve $\gamma\colon S^1\injects M$ which is not contained in an embedded three-ball and does not intersect the boundary. There is a tubular neighbourhood $T\subset M$ of $\gamma(S^1)$ whose boundary does not intersect the boundary of $M$. Denote by $N$ the compact three-manifold obtained from $M$ by removing $T$ and by $\nu\colon N\injects M$ the canonical inclusion. We claim that $\pi_1(M)$ is a subgroup of $\pi_1(M\cup_{\partial M}N)$ and that $M\cup_{\partial M}N$ remains irreducible, from which the result will follow.

We first show that $\pi_1(M)$ is a subgroup $\pi_1(M\cup_{\partial M}N)$. We have a canonical inclusion $M\injects M\cup_{\partial M}N$ and a map $M\cup_{\partial M}N\to M$ defined by the identity on $M$ and $\nu$ on $N$. As the composition of these two maps is the identity, we see that $\pi_1(M)$ is a subgroup of $\pi_1(M\cup_{\partial M}N)$. In particular, $\pi_1(M)$ is a retract of $\pi_1(M\cup_{\partial M}N)$.

We now show that $M\cup_{\partial M}N$ is irreducible. First note that as the image of $\gamma$ is not contained in an embedded ball, it follows that $N$ is also irreducible. Suppose that $M\cup_{\partial M}N$ is not irreducible. Then there is a sphere $S^2$, embedded in the interior of $M\cup_{\partial M}N$, that does not bound a ball. After an isotopy, we may assume that $S^2$ is transverse to $\partial M$. Thus, $S^2\cap \partial M$ is either empty or consists of embedded circles. If the intersection is non-empty, take some innermost circle $S^1\subset S^2\cap \partial M$. Then the two-disc $D$ it bounds in $S^2$ must embed in $M$ or $N$. As $M$ had incompressible boundary, it follows that we may isotope $D$ through $\partial M$ and reduce the number of components in $S^2\cap \partial M$. Continuing in this way, we see that we may isotope $S^2$ so that $S^2\cap \partial M$ is empty. If $S^2\cap \partial M$ is empty, then $S^2$ is contained in $M$ or $N$. As $M$ and $N$ are both irreducible, it follows that $M\cup_{\partial M}N$ is also irreducible. Since $N$ is irreducible and $\pi_1(N)$ is infinite, it follows that $N$ is aspherical.
\end{proof}

\begin{proof}[Proof of \cref{manifold_free-by-cyclic}]
By \cref{closure,facts}, we may assume that $M$ is compact, orientable, aspherical and with incompressible boundary. As subgroups of virtually free-by-cyclic groups are virtually free-by-cyclic, we may also assume that $M$ has non-empty purely toroidal boundary by \cref{drilled}. If $M$ is a hyperbolic three-manifold, then it is virtually fibred by \cite[Theorem 17.14]{wise_21_quasiconvex} and \cite[Theorem 1.1]{Agol2008}. If $M$ is a graph manifold, it is virtually fibred by \cite{wang_97}. In all other cases, it is virtually fibred by \cite[Corollary 1.3]{przytycki_18}. As $M$ has non-empty toroidal boundary, it virtually fibres with fibre a surface with non-empty boundary. Thus, $G$ is virtually free-by-cyclic.
\end{proof}

\begin{proof}[Proof of \cref{coherence}]
We first make two general observations about coherence of group rings. Let $G$ be a group and $R$ a ring. The first observation is that if $H\leqslant G$ is a finite index subgroup with $RH$ coherent, then $RG$ is coherent. The second is that if $RH$ is coherent for all finitely generated subgroups $H\leqslant G$, then $RG$ is coherent. We now specialise to the case where $G$ is the fundamental group of a three-manifold group with $\cd_{\Q}(G)<3$ and where $R = K$ is a field. By the above two observations and \cref{manifold_free-by-cyclic}, it suffices to prove the result in the case that $G$ is free-by-cyclic. This is precisely \cite[Theorem 3.3]{JZL23}.
\end{proof}

\subsection{An example}

We now present an example of a locally virtually free-by-cyclic three-manifold group that is not virtually free-by-cyclic, showing that \cref{manifold_free-by-cyclic} is sharp. A straightforward example would be a three-manifold whose fundamental group has no bound on the order of its torsion elements. Instead, we construct an example that is Seifert fibred with base orbifold of infinite type and whose fundamental group is torsion-free.

We first require a couple of facts about Seifert fibred spaces. The reader is directed towards \cite[Section 10]{martelli_22} for the necessary background. 

Let $M$ be a compact orientable three-manifold and let $M\to S$ be a Seifert fibration. If $M$ fibres over the circle, then $M$ admits a foliation by compact surfaces. Note that if $M$ has non-trivial boundary, then we may glue solid tori along the boundary and extend this foliation. So now we may apply work of Eisenbud--Hirsch--Neumann \cite[Theorems 3.4 \& 6.1]{eisenbud_81} to conclude that if $S$ has genus at least two, then such a foliation must be homotopic to a foliation transverse to the fibres of the Seifert fibration and moreover exists only if $e(M) = 0$, where $e(M)$ denotes the Euler number of the Seifert fibration. Recall that the Euler number of a Seifert fibred space with boundary is only defined modulo the integers.

Now let $M_n$ denote the orientable Seifert fibred space with Euler number $e(M) = 1/n$  whose base orbifold is a genus two surface with two boundary components and a single cone point. The Euler characteristic of $M_n$ is zero, so any finite index subgroup of $\pi_1(M_n)$ that is free-by-cyclic, must be \{finitely generated free\}-by-cyclic by work of Feighn--Handel \cite{feighn_99}. In particular, by Stallings' Fibration Theorem \cite{stallings_62}, the corresponding cover of $M_n$ must fibre over the circle. Since any finite degree cover of $M_n$ induces a finite degree orbifold cover of the base orbifold, we see that the minimal degree of a cover of $M_n$ with Euler number zero is $n$. So by the previous paragraph, the minimal degree of a cover of $M_n$ that fibres over the circle is $n$.

Consider the three-manifold 
\[
M = M_2\cup_T M_3\cup_T\ldots
\]
where $T$ is one of the boundary tori of $M_i$. If $\pi_1(M)$ were virtually free-by-cyclic, then this would imply that every $M_n$ admits a finite cover of uniformly bounded index that fibres over the circle. But this cannot happen by the previous paragraph and so $\pi_1(M)$ is not virtually free-by-cyclic; it is locally virtually free-by-cyclic by \cref{manifold_free-by-cyclic}.

\section{Linnell's class $\mathcal{C}$}
\label{C_section}

In this section we show that Linnell's class $\mathcal{C}$ is closed under free products and use this to show that all three-manifold groups lie in $\mathcal{C}$.

We first remind the reader of the definition of elementary amenable groups. Denote by $\mathcal{EA}_0$ the class of groups that are abelian or finite. For each ordinal $\alpha$, we define $\mathcal{EA}_\alpha$ to consist of extensions of groups in $\mathcal{EA}_{\beta}$ for $\beta < \alpha$, and all directed unions of groups, each of which lies in some $\mathcal{EA}_\beta$ with $\beta<\alpha$. The union of all the classes $\mathcal{EA}_\alpha$ is precisely the class of elementary amenable groups.

We may similarly stratify Linnell's class $\mathcal{C}$. Denote by $\mathcal{C}_0$ the class of free groups. For each ordinal $\alpha$, let $\mathcal{C}_{\alpha}$ consist of all elementary amenable extensions of groups in $\mathcal{C}_{\beta}$ for $\beta < \alpha$, and all directed unions of groups, each of which lies in some $\mathcal C_\beta$ with $\beta<\alpha$. It is clear that the class $\mathcal C$ is precisely the union of the classes $\mathcal C_\alpha$ taken over all ordinals $\alpha$.

It is easy to see that every  $\mathcal{EA}_\alpha$ and every $\mathcal C_\alpha$ are closed under taking subgroups.

\begin{proof}[Proof of \cref{C_free_products}]
	Let $\alpha$ be an ordinal. Consider a free product $*_{i\in I}G_i$ where for all $i$, $G_i\in \mathcal{C}_{\alpha}$. We claim that $*_{i\in I}G_i\in\mathcal{C}$. The proof is by transfinite induction. Since free products of free groups are free, the base case holds.
	
	Now consider an ordinal $\alpha>0$, and suppose that the claim is true for all ordinals $\beta$ with $\beta < \alpha$.
	Let $J \subseteq I$ be such that for every $i \not\in J$ we have
\[
1\to N_i\to G_i\to A_i\to 1
\]
with $N_i\in \mathcal{C}_{\beta_i}$ where $\beta_i < \alpha$, and with $A_i$ elementary amenable (and possibly trivial), and for every $i \in J$ we have $G_i = \bigcup_{j\in J_i}G_{i, j}$ with $G_{i, j}\in \mathcal{C}_{\beta_{i,j}}$ for some $\beta_{i,j}<\alpha$. For $\beta < \alpha$ let $G_{i,\beta}$ denote the union of the subgroups $G_{i,j}$ such that $\beta_{i,j}< \beta$; if no such subgroup exists, we take $G_{i,\beta}$ to be the trivial group. By definition, $G_{i,\beta} \in \mathcal C_\beta$.

 Consider the homomorphism
	\[
	*_{i\in I}G_i\to \bigoplus_{i\in I - J}A_i
	\]
	obtained from the homomorphisms above in the obvious way, and with groups $G_i$ for $i \in J$ lying in the kernel.
	Note that the image is a subgroup of a direct sum of elementary amenable groups, and hence an elementary amenable group itself.
	By the Kurosh Subgroup Theorem, the kernel $K$ is a free product of conjugates of the groups $N_i$, conjugates of the groups $G_i$ with $i \in J$, and a free group $F$.
	
	Take $\beta<\alpha$. Let $K_\beta$ denote the free product of $F$ with the conjugates of the groups $N_i$ that lie in $\mathcal C_\beta$, and with conjugates of groups $G_{i,\beta}$ for $i \in J$. By the inductive hypothesis, $K_\beta$ lies in $\mathcal C$. Note that $K = \bigcup_{\beta<\alpha} K_\beta$, and hence $K$ lies in $\mathcal C$. But then  
	$*_{i\in I}G_i$ lies in $\mathcal C$ as well, being an extension of $K$ by an elementary amenable group.
	
	We finish by observing that every free product of groups in $\mathcal C$ is a free product of groups lying in $\mathcal C_\alpha$ for some $\alpha$, since $\mathcal C$ is the union of the classes $\mathcal C_\alpha$, and ordinals are closed under taking unions.
\end{proof}

We need one final result due to Friedl--L\"uck and then we will be ready to prove \cref{main thm}.

\begin{theorem}[Friedl--L\"uck {\cite[Theorem 3.2(3)]{friedl_19}}]
\label{aspherical_C}
Let $M$ be a closed, connected, orientable and aspherical three-manifold. The fundamental group $G$ of $M$ lies in Linnell's class $\mathcal C$.
\end{theorem}

\begin{proof}[Proof of \cref{main thm}]
	Let $M$ be a connected three-manifold with fundamental group $G$.
	
	Suppose first that $G$ is finitely generated. Since class $\mathcal{C}$ is closed under extensions by finite groups and under free products by \cref{C_free_products}, we may assume that $M$ is compact, orientable, aspherical and with incompressible boundary by \cref{facts}. If $M$ has empty boundary, we are done by \cref{aspherical_C}. Otherwise, we are done by \cref{manifold_free-by-cyclic}, since $M$ is homotopy equivalent to its spine, which is an aspherical two-complex, and hence $\cd_{\Q}(G) <3$.
	
	Finally, suppose that $\pi_1(M)$ is not finitely generated. Let $M_0\to M_1\to \ldots\to M$ be a sequence of covers such that $\pi_1(M_i)$ is finitely generated for all $i$ and such that $\bigcup_{i\geqslant 0}\pi_1(M_i) = \pi_1(M)$. Since we showed above that $\pi_1(M_i)$ is in $\mathcal{C}$ for all $i$, the group $\pi_1(M)$ is also in $\mathcal{C}$.	
\end{proof}

\begin{proof}[Proof of \cref{main cor}]
	Linnell \cite{linnell_93}*{Theorem 1.5} proved that all groups in $\mathcal C$ with a uniform bound on cardinalities of torsion subgroups satisfy the Strong Atiyah Conjecture over $\C$. It is very easy to see that torsion-free groups satisfying this conjecture do not have non-trivial zero-divisors, see \cite{luck_02}*{Lemma 10.15}. If $G$ is not torsion-free, but is finitely generated, \cref{facts} implies that $G$ is virtually torsion-free. Being virtually torsion-free gives a bound on the size of torsion subgroups. This finishes the proof.
\end{proof}

\subsection{Further properties of $\mathcal{C}$}

Schick proved in \cite{schick_00} (see also \cite{schick_02}) that class $\mathcal{D}$, a class containing the torsion-free groups from $\mathcal{C}$, is closed under free products. However, the argument relies on class $\mathcal{D}$ being closed under direct sums, which class $\mathcal{C}$ is not by the following lemma.

\begin{lemma}
If $G = A\times B\in \mathcal{C}$, then $A, B\in \mathcal{C}$ and at least one of $A$ or $B$ is elementary amenable.
\end{lemma}

\begin{proof}
Class $\mathcal{C}$ is closed under taking subgroups, so certainly $A, B\in \mathcal{C}$. Suppose now that $A$ and $B$ are not elementary amenable. As elementary amenable groups are closed under directed unions and extensions, it follows that both $A$ and $B$ must contain a non-abelian free subgroup. In particular, $G$ contains a copy of $F_2\times F_2$. 

Let $\alpha$ be the first ordinal such that $\mathcal{C}_{\alpha}$ contains a group $K$ containing a copy of $F_2\times F_2$. If $K$ is a directed union of groups, each of which lies in some $C_{\beta}$ with $\beta<\alpha$, then since $F_2\times F_2$ is finitely generated, it appears as a subgroup of one of the groups in this directed union, contradicting minimality of $\alpha$. If $K$ is an extension of a group $N$ in $C_{\beta}$, for $\beta<\alpha$, by an elementary amenable group $Q$, then we claim that $N$ must also contain a copy of $F_2\times F_2$. Indeed let $L$ and $R$ be the kernels of the induced maps from the two $F_2$ factors to $Q$. As $Q$ is elementary amenable and non-trivial normal subgroups of non-abelian free groups are non-abelian free, we must have that $L$ and $R$ are non-abelian free groups. Since $L\times R\leqslant N$, the claim follows. But this also contradicts the minimality of $\alpha$. So no group in $\mathcal{C}$ can contain $F_2\times F_2$, a contradiction.
\end{proof}

Recall that a group has property \emph{$\mathrm{FAb}$} if its finite-index subgroups have finite abelianisations. Property $\mathrm{FAb}$ is a consquence of Property $T$. It is immediate that \emph{$\mathrm{FAb}$} passes to finite-index subgroups and overgroups. 

\begin{lemma}
	If $G$ is a finitely generated group with $\mathrm{FAb}$ lying in  $\mathcal C$ then every elementary amenable quotient of $G$ is finite. Therefore, $G$ is itself finite.
\end{lemma}   
\begin{proof}
	We will argue by contradiction.
	Take $\alpha$ to be the smallest ordinal such that there exists $Q \in \mathcal{EA}_\alpha$ that is infinite and fits into a short exact sequence
	\[1\to N \to G \to Q\to1 \]
	with $G$ a finitely generated group in $\mathcal C$ with property $\mathrm{FAb}$. 
	
	Since $Q$ is finitely generated and infinite, it cannot be abelian as $G$ has $\mathrm{FAb}$. Hence we see that $\alpha \neq 0$ and so $Q$ is an extension
	\[
	1\to Q_0 \to Q \to Q_1\to1
	\]
	with $Q_0$ and $Q_1$ lying in $Q \in \mathcal{EA}_\beta$ for some $\beta < \alpha$.
 	Thus $G$ maps onto $Q_1$ with kernel an extension of $N$ by an elementary amenable group. This forces the kernel to lie in $\mathcal C$, and hence $Q_1$ is finite by minimality of $\alpha$. But then a finite-index subgroup of $G$ fits into an exact sequence with kernel $N$ and quotient $Q_0$. By minimality of $\alpha$, the group $Q_0$ is finite as well. This proves that $Q$ was also finite, a contradiction.
	
	\smallskip
	We will now prove the last claim. Again, this will be done by contradiction. Suppose that $\gamma$ is the smallest ordinal such that $\mathcal C_\gamma$ contains an infinite, finitely generated group $G$ with $\mathrm{FAb}$. Since non-trivial free groups do not have $\mathrm{FAb}$, and since $G$ is finitely generated, the group $G$ must fit into a short exact sequence 
	\[
	1\to N \to G \to Q \to 1
	\] 
	with $N$ in $\mathcal C_\delta$ for some $\delta<\gamma$, and with $Q$ elementary amenable. But then $Q$ must be finite, and hence $N$ has property $\mathrm{FAb}$. This contradicts the minimality of $\gamma$.
\end{proof}

\bibliographystyle{amsalpha}
\bibliography{bibliography}

\begin{bibdiv}
\begin{biblist}

\bib{aschenbrenner_15}{book}{
      author={Aschenbrenner, Matthias},
      author={Friedl, Stefan},
      author={Wilton, Henry},
       title={3-manifold groups},
      series={EMS Series of Lectures in Mathematics},
   publisher={European Mathematical Society (EMS), Z\"{u}rich},
        date={2015},
        ISBN={978-3-03719-154-5},
         url={https://doi-org.ezproxy-prd.bodleian.ox.ac.uk/10.4171/154},
      review={\MR{3444187}},
}

\bib{Agol2008}{article}{
      author={Agol, Ian},
       title={Criteria for virtual fibering},
        date={2008},
        ISSN={1753-8416},
     journal={J. Topol.},
      volume={1},
      number={2},
       pages={269\ndash 284},
  url={https://doi-org.ezproxy-prd.bodleian.ox.ac.uk/10.1112/jtopol/jtn003},
      review={\MR{2399130}},
}

\bib{baumslag_09}{article}{
      author={Baumslag, Gilbert},
      author={Fine, Benjamin},
      author={Miller, Charles~F., III},
      author={Troeger, Douglas},
       title={Virtual properties of cyclically pinched one-relator groups},
        date={2009},
        ISSN={0218-1967},
     journal={Internat. J. Algebra Comput.},
      volume={19},
      number={2},
       pages={213\ndash 227},
  url={https://doi-org.ezproxy-prd.bodleian.ox.ac.uk/10.1142/S0218196709005032},
      review={\MR{2512551}},
}

\bib{eisenbud_81}{article}{
      author={Eisenbud, David},
      author={Hirsch, Ulrich},
      author={Neumann, Walter},
       title={Transverse foliations of {S}eifert bundles and self-homeomorphism
  of the circle},
        date={1981},
        ISSN={0010-2571},
     journal={Comment. Math. Helv.},
      volume={56},
      number={4},
       pages={638\ndash 660},
  url={https://doi-org.ezproxy-prd.bodleian.ox.ac.uk/10.1007/BF02566232},
      review={\MR{656217}},
}

\bib{feighn_99}{article}{
      author={Feighn, Mark},
      author={Handel, Michael},
       title={Mapping tori of free group automorphisms are coherent},
        date={1999},
        ISSN={0003-486X},
     journal={Ann. of Math. (2)},
      volume={149},
      number={3},
       pages={1061\ndash 1077},
         url={https://ezproxy-prd.bodleian.ox.ac.uk:2102/10.2307/121081},
      review={\MR{1709311}},
}

\bib{friedl_19}{article}{
      author={Friedl, Stefan},
      author={L\"{u}ck, Wolfgang},
       title={{$L^2$}-{E}uler characteristics and the {T}hurston norm},
        date={2019},
        ISSN={0024-6115},
     journal={Proc. Lond. Math. Soc. (3)},
      volume={118},
      number={4},
       pages={857\ndash 900},
  url={https://doi-org.ezproxy-prd.bodleian.ox.ac.uk/10.1112/plms.12202},
      review={\MR{3938714}},
}

\bib{grigorchuk_00}{article}{
      author={Grigorchuk, Rostislav~I.},
      author={Linnell, Peter},
      author={Schick, Thomas},
      author={\.{Z}uk, Andrzej},
       title={On a question of {A}tiyah},
        date={2000},
        ISSN={0764-4442},
     journal={C. R. Acad. Sci. Paris S\'{e}r. I Math.},
      volume={331},
      number={9},
       pages={663\ndash 668},
         url={https://doi.org/10.1016/S0764-4442(00)01702-X},
      review={\MR{1797748}},
}

\bib{JZL23}{misc}{
      author={Jaikin-Zapirain, Andrei},
      author={Linton, Marco},
       title={On the coherence of one-relator groups and their group algebras},
        date={2023},
         url={https://arxiv.org/abs/2303.05976},
        note={arXiv:2303.05976},
}

\bib{linnell_93}{article}{
      author={Linnell, Peter~A.},
       title={Division rings and group von {N}eumann algebras},
        date={1993},
        ISSN={0933-7741},
     journal={Forum Math.},
      volume={5},
      number={6},
       pages={561\ndash 576},
  url={https://ezproxy-prd.bodleian.ox.ac.uk:2102/10.1515/form.1993.5.561},
      review={\MR{1242889}},
}

\bib{linnell_07}{article}{
      author={Linnell, Peter},
      author={Schick, Thomas},
       title={Finite group extensions and the {A}tiyah conjecture},
        date={2007},
        ISSN={0894-0347,1088-6834},
     journal={J. Amer. Math. Soc.},
      volume={20},
      number={4},
       pages={1003\ndash 1051},
         url={https://doi.org/10.1090/S0894-0347-07-00561-9},
      review={\MR{2328714}},
}

\bib{luck_02}{book}{
      author={L{\"{u}}ck, Wolfgang},
       title={{$L^2$}-invariants: theory and applications to geometry and
  {$K$}-theory},
      series={Ergebnisse der Mathematik und ihrer Grenzgebiete. 3. Folge. A
  Series of Modern Surveys in Mathematics [Results in Mathematics and Related
  Areas. 3rd Series. A Series of Modern Surveys in Mathematics]},
   publisher={Springer-Verlag, Berlin},
        date={2002},
      volume={44},
        ISBN={3-540-43566-2},
  url={https://ezproxy-prd.bodleian.ox.ac.uk:2102/10.1007/978-3-662-04687-6},
      review={\MR{1926649}},
}

\bib{martelli_22}{article}{
      author={Martelli, Bruno},
       title={An introduction to geometric topology},
        date={2022-04},
      eprint={1610.02592},
}

\bib{przytycki_18}{article}{
      author={Przytycki, Piotr},
      author={Wise, Daniel~T.},
       title={Mixed 3-manifolds are virtually special},
        date={2018},
        ISSN={0894-0347},
     journal={J. Amer. Math. Soc.},
      volume={31},
      number={2},
       pages={319\ndash 347},
         url={https://doi-org.ezproxy-prd.bodleian.ox.ac.uk/10.1090/jams/886},
      review={\MR{3758147}},
}

\bib{schick_00}{article}{
      author={Schick, Thomas},
       title={Integrality of {$L^2$}-{B}etti numbers},
        date={2000},
        ISSN={0025-5831},
     journal={Math. Ann.},
      volume={317},
      number={4},
       pages={727\ndash 750},
  url={https://doi-org.ezproxy-prd.bodleian.ox.ac.uk/10.1007/PL00004421},
      review={\MR{1777117}},
}

\bib{schick_02}{article}{
      author={Schick, Thomas},
       title={Erratum: ``{I}ntegrality of {$L^2$}-{B}etti numbers''},
        date={2002},
        ISSN={0025-5831},
     journal={Math. Ann.},
      volume={322},
      number={2},
       pages={421\ndash 422},
  url={https://doi-org.ezproxy-prd.bodleian.ox.ac.uk/10.1007/s002080100282},
      review={\MR{1894160}},
}

\bib{schreve_14}{article}{
      author={Schreve, Kevin},
       title={The strong {A}tiyah conjecture for virtually cocompact special
  groups},
        date={2014},
        ISSN={0025-5831},
     journal={Math. Ann.},
      volume={359},
      number={3-4},
       pages={629\ndash 636},
  url={https://ezproxy-prd.bodleian.ox.ac.uk:2102/10.1007/s00208-014-1007-9},
      review={\MR{3231009}},
}

\bib{Scott1973}{article}{
      author={Scott, G.~P.},
       title={Compact submanifolds of {$3$}-manifolds},
        date={1973},
        ISSN={0024-6107},
     journal={J. London Math. Soc. (2)},
      volume={7},
       pages={246\ndash 250},
  url={https://doi-org.ezproxy-prd.bodleian.ox.ac.uk/10.1112/jlms/s2-7.2.246},
      review={\MR{326737}},
}

\bib{stallings_62}{inproceedings}{
      author={Stallings, John},
       title={On fibering certain {$3$}-manifolds},
        date={1962},
   booktitle={Topology of 3-manifolds and related topics ({P}roc. {T}he {U}niv.
  of {G}eorgia {I}nstitute, 1961)},
   publisher={Prentice-Hall, Englewood Cliffs, N.J.},
       pages={95\ndash 100},
      review={\MR{0158375}},
}

\bib{wise_21_quasiconvex}{book}{
      author={Wise, Daniel~T.},
       title={The structure of groups with a quasiconvex hierarchy},
      series={Annals of Mathematics Studies},
   publisher={Princeton University Press, Princeton, NJ},
        date={2021},
      volume={209},
        ISBN={[9780691170442]; [9780691170459]; [9780691213507]},
      review={\MR{4298722}},
}

\bib{wang_97}{article}{
      author={Wang, Shicheng},
      author={Yu, Fengchun},
       title={Graph manifolds with non-empty boundary are covered by surface
  bundles},
        date={1997},
        ISSN={0305-0041},
     journal={Math. Proc. Cambridge Philos. Soc.},
      volume={122},
      number={3},
       pages={447\ndash 455},
  url={https://doi-org.ezproxy-prd.bodleian.ox.ac.uk/10.1017/S0305004197001709},
      review={\MR{1466648}},
}

\end{biblist}
\end{bibdiv}

\end{document}